\documentclass[11pt]{amsart}
\newcommand{\preprint}[1]{}

\newcommand{\hide}[1]{}

\usepackage{amssymb}
\usepackage{amsbsy}
\usepackage{amscd}
\usepackage{amsmath}
\usepackage{amsthm}

\input xy
\xyoption{all}

\numberwithin{equation}{section}

\theoremstyle{plain}
\newtheorem{thm}{Theorem}[section]

\newtheorem{lem}[thm]{Lemma}

\theoremstyle{definition}
\newtheorem{defi}[thm]{Definition}

\newtheorem{example}[thm]{Example}

\newtheorem{rem}[thm]{Remark}

\newcommand{\C}{{\mathcal C}}
\newcommand{\D}{{\mathcal D}}

\newcommand{\F}{{\mathcal F}}
\newcommand{\FF}{{\mathbb F}}

\newcommand{\GG}{{\mathbb G}}
\newcommand{\Gerbes}{{\rm Gerbes}}

\newcommand{\HH}{{\mathbb H}}

\newcommand{\LB}{{\mathcal L}}

\newcommand{\fM}{{\mathfrak M}}

\newcommand{\Out}{{\rm Out}}
\newcommand{\SheafOut}{{\mathcal Out}}
\renewcommand{\P}{{\mathcal P}}
\newcommand{\PP}{{\mathbb P}}

\newcommand{\U}{{\mathcal U}}
\newcommand{\V}{{\mathcal V}}

\newcommand{\X}{{\mathcal X}}
\newcommand{\Y}{{\mathcal Y}}
\newcommand{\Z}{{\mathcal Z}}
\newcommand{\RealNumbers}{{\mathbb R}}
\newcommand{\Integers}{{\mathbb Z}}
\newcommand{\ComplexNumbers}{{\mathbb C}}

\newcommand{\RightArrowOf}[1]{\stackrel{#1}{\rightarrow}}

\newcommand{\restricted}[2]{#1_{\mid_{#2}}}

\newcommand{\rank}{{\rm rank}}

\newcommand{\Hom}{{\rm Hom}}
\newcommand{\Isom}{{\rm Isom}}
\newcommand{\Aut}{{\rm Aut}}
\newcommand{\Band}{{\rm Band}}
\newcommand{\Diff}{{\rm Diff}}

\newcommand{\SheafHom}{{\mathcal H}om}

\newcommand{\SheafAut}{{\mathcal Aut}}
\newcommand{\SheafIsom}{{\mathcal Isom}}

\begin{document}
\title[Universal families of marked compact hyperk\"{a}hler manifolds]
{On the existence of universal families of marked irreducible holomorphic symplectic manifolds}
\author{Eyal Markman}
\address{Department of Mathematics and Statistics, 
University of Massachusetts, Amherst, MA 01003, USA}
\email{markman@math.umass.edu}
\date{\today}

\begin{abstract}
We prove the existence of a global family with natural universal properties over every component of the moduli space of marked irreducible holomorphic symplectic manifolds. The analogous  result follows for the Teichm\"{u}ller spaces.
\end{abstract}

\maketitle
\tableofcontents

%
\section{Introduction}

An {\em irreducible holomorphic symplectic manifold} is a simply connected compact K\"{a}hler manifold $X$, such that 
$H^0(X,\Omega^2_X)$ is one dimensional spanned by an everywhere non-degenerate holomorphic $2$-form
\cite{beauville-varieties-with-zero-c-1,huybrechts-basic-results}. 
The second integral cohomology of $X$ is endowed with a non-degenerate integral symmetric bilinear pairing of signature $(3,b_2-3)$,
where $b_2$ is the second Betti number of $X$. The pairing is known as the {\em Beauville-Bogomolov-Fujiki} pairing.
A {\em marking} for $X$ is an isometry $\eta:H^2(X,\Integers)\rightarrow \Lambda$ with a fixed lattice $\Lambda$. 
An isomorphism of two marked pairs $(X_i,\eta_i)$, $i=1,2$, consists of an isomorphism $f:X_1\rightarrow X_2$, such that 
$\eta_1\circ f^*=\eta_2$. 

Given an analytic space $B$ and a discrete group $H$, denote by $\underline{H}_B$ the trivial local system over $B$ with fiber $H$. 
A {\em family of marked irreducible holomorphic symplectic manifolds} over an analytic base $B$ consists of
a family $\pi:\X\rightarrow B$ of such manifolds, together with an isometric trivialization 
$\eta:R^2\pi_*\Integers\rightarrow \underline{\Lambda}_B$. We will call the pair $(\pi,\eta)$  a {\em $\Lambda$-marked} family for short.
Two $\Lambda$-marked families $(\pi:\X\rightarrow B,\eta)$ and $(\tilde{\pi}:\tilde{\X}\rightarrow B,\tilde{\eta})$ 
are {\em isomorphic}, if there exists an isomorphism $f:\X\rightarrow \tilde{\X}$, such that $\tilde{\pi}f=\pi$ and 
$\tilde{\eta}=\eta\circ f^*$, where $f^*:R^2\tilde{\pi}_*\Integers\rightarrow R^2\pi_*\Integers$ is the isomorphism induced by $f$. 
Given a marked family $(\pi:\X\rightarrow B,\eta)$ and a morphism $\kappa:\tilde{B}\rightarrow B$, we get the pulled back family
$\kappa^*(\pi):\X\times_B\tilde{B}\rightarrow \tilde{B}$ with the marking $\kappa^*(\eta)$. 
\hide{
We get a category $\C$, whose objects are such marked families $(\pi:\X\rightarrow B,\eta)$. 
A morphism in this category from $(\pi:\X\rightarrow B,\eta)$ to $(\tilde{\pi}:\tilde{\X}\rightarrow \tilde{B},\tilde{\eta})$
consists of a pair $(\mu,\tilde{\mu})$ of a morphism fitting in a cartesian diagram
\[
\xymatrix{
\X\ar[d]_{\pi}\ar[r]^{\tilde{\mu}}
&
\tilde{\X}\ar[d]^{\tilde{\pi}}
\\
B\ar[r]_{\mu} & \tilde{B}.
}
\]
}
Let $\F_\Lambda$ be the functor, 
from the category of analytic spaces to the category of sets, which associates to an analytic space $B$ the
set of isomorphism classes of $\Lambda$-marked families $(\pi,\eta)$ over $B$.
There exists a non-Hausdorff (disconnected)
complex manifold $\fM_\Lambda$ of dimension $\rank(\Lambda)-2$, 
which coarsely represents $\F_\Lambda$ \cite{huybrechts-basic-results}. 

Let $X$ be an irreducible holomorphic symplectic manifold. Denote by $\Aut_0(X)$ the subgroup, of the auotomorphism group 
$\Aut(X)$ of $X$, consisting of elements which act trivially on $H^2(X,\Integers)$. Huybrechts proved that 
$\Aut_0(X)$ is a finite group \cite[Prop. 9.1(v)]{huybrechts-basic-results}.
Associated to every family $\pi:\X\rightarrow B$ of irreducible holomorphic symplectic manifolds we have a local system
$\SheafAut_0(\pi)$ over $B$, whose fiber over $b\in B$ is the group $\Aut_0(X_b)$ of the fiber $X_b$ of $\pi$ over $b$,
by \cite[Theorem 2.1]{hassett-tschinkel-lagrangian-planes}. 

Given an $\SheafAut_0(\pi)$ torsor $\P$ over $B$, we get the family $\tilde{\X}:=\X\times^{\SheafAut_0(\pi)}\P$. 
Denote by $\tilde{\pi}:\tilde{\X}\rightarrow B$ the natural projection. Note that the local systems
$R^2\pi_*\Integers$ and $R^2\tilde{\pi}_*\Integers$ are naturally isomorphic, and so 
a marking $\eta$ for the former induces a marking  
$\tilde{\eta}$
for the latter.
We denote this new marked family by 
$(\tilde{\pi},\tilde{\eta}):=(\pi,\eta)\times\P$.
Two $\Lambda$-marked families $(\pi:\X\rightarrow B,\eta)$ and $(\tilde{\pi}:\tilde{\X}\rightarrow B,\tilde{\eta})$ 
are said to be {\em equivalent}, if there exists an $\SheafAut_0(\pi)$ torsor $\P$ over $B$, such that 
$(\tilde{\pi},\tilde{\eta})=(\pi,\eta)\times\P$. 
The map $\P\mapsto (\pi,\eta)\times\P$ is a bijection between the set 
$H^1(B,\SheafAut_0(\pi))$, of isomorphism classes of $\SheafAut_0(\pi)$ torsors, and the set of isomorphism classes of 
$\Lambda$-marked families equivalent to $(\pi,\eta)$.
Its inverse sends $(\tilde{\pi},\tilde{\eta})$ to 
the isomorphism class of the $\SheafAut_0(\pi)$ torsor 
$\SheafIsom((\pi,\eta),(\tilde{\pi},\tilde{\eta}))$  of local isomorphisms of the two families compatible with the markings. 
We elaborate on this bijection in Remark \ref{rem-principal-G-bunle-invariant}.

Let $\underline{\Lambda}$ be the trivial local system over $\fM_\Lambda$ with fiber $\Lambda$.
The main result of this paper is the following statement.

\begin{thm}
\label{thm-main}
There exists  a family $\pi:\X\rightarrow \fM_\Lambda$ of irreducible holomorphic symplectic manifolds and a marking 
$\eta:R^2\pi_*\Integers\rightarrow \underline{\Lambda}$ satisfying the following universal property.
Given a family of $\Lambda$-marked irreducible holomorphic symplectic manifolds $(\tilde{\pi}:\tilde{\X}\rightarrow B,\tilde{\eta})$
over an analytic space $B$, the pullback $(\kappa^*(\pi),\kappa^*(\eta))$ via the classifying morphism 
$\kappa\nolinebreak :\nolinebreak B\nolinebreak \rightarrow \nolinebreak \fM_\Lambda$
is equivalent to $(\tilde{\pi},\tilde{\eta})$.
The marked family $(\pi,\eta)$ satisfying this property is unique, up to isomorphism. $\SheafAut_0(\pi)$ restricts to each connected component of $\fM_\Lambda$ as a trivial local system.
\end{thm}

The Theorem is proved in section \ref{sec-triviality-of-a-gerbe}.
The Teichm\"{u}ller space of an irreducible holomorphic symplectic manifold $X$ maps to the moduli space of marked pairs \cite[Cor. 4.31]{verbitsky}, and so the universal family over the latter pulls back to one over the Teichm\"{u}ller space.

Consider the moduli space $\fM_{\Lambda,G}$ of isomorphism classes of triples $(X,\eta,\psi)$, where $(X,\eta)$ is a marked
irreducible holomorphic symplectic manifold, and $\psi:\Aut_0(X)\rightarrow G$ is an isomorphism with a fixed finite group $G$.
The automorphism group $\Aut(X,\eta)$ of the marked pair $(X,\eta)$ is $\Aut_0(X)$, and the automorphism group 
$\Aut(X,\eta,\psi)$ of the triple is the center of $\Aut_0(X)$. 

\begin{rem}
\label{rem-fine-moduli-space} 
If the center of $G$ is trivial, then 
the moduli space $\fM_{\Lambda,G}$ represents the functor 
from the category of analytic spaces to the category of sets, which associates to an analytic space $B$ the set of isomorphism classes of triples $(\pi,\eta,\psi)$, consisting of a
family $\pi \nolinebreak :  \nolinebreak \X\rightarrow \nolinebreak B$ 
of irreducible holomorphic symplectic manifolds,  an isometric trivialization 
$\eta\nolinebreak : \nolinebreak R^2\pi_*\Integers\nolinebreak \rightarrow \nolinebreak \underline{\Lambda}_B$, and a trivialization
$\psi:\SheafAut_0(\pi)\rightarrow \underline{G}_B$. The local universal families glue in this case uniquely to a global universal family.
\end{rem}

Let us sketch the proof of Theorem \ref{thm-main}.
We have the forgetful morphism $\phi:\fM_{\Lambda,G}\rightarrow \fM_\Lambda$. 
We first show that $\phi$ restricts to each connected component $\fM^0_{\Lambda,G}$ of $\fM_{\Lambda,G}$ as an isomorphism onto the corresponding connected component $\fM^0_\Lambda$ of $\fM_\Lambda$ (Lemma \ref{lemma-M-Lambda-G}). 

\hide{
We have a global family $\Y\rightarrow \fM^0_{\Lambda,G}$, whose fiber over the isomorphism class of 
$(X,\eta,\psi)$ is the the quotient of $X$ by the center of $\Aut_0(X)$. 
Hence, such a family exists over $\fM^0_\Lambda$ as well.
The existence of a universal family of irreducible holomorphic symplectic manifolds over $\fM^0_\Lambda$ is 
thus analogous to the following lifting problem.
Consider a short exact sequence of groups
\[
0\rightarrow A \rightarrow G \rightarrow Q\rightarrow 0,
\]
where $G$ is a reductive group and $A$ is an abelian subgroup. 
Let $P$ be a principal $Q$-bundle over a topological space $B$. 
Denote by $\underline{A}_B$ the sheaf of continuous sections of the trivial group bundle with fiber $A$. 
Then the obstruction to lifting $P$ to a principal $G$-bundle over $B$ corresponds to a class $[P]$
in the \v{C}ech cohomology group $\check{H}^2(B,\underline{A}_B)$. 
Lifts exist locally, and the collection of all local lifts corresponds to the categorical construct of a {\em gerbe with band}
$\underline{A}_B$. The obstruction class $[P]$ in $\check{H}^2(B,\underline{A}_B)$ is the equivalence class of this gerbe \cite[Prop. 5.2.3 and Theorem 5.2.9]{giraud,brylinski}. If the groups $A$, $G$, and $Q$ are complex and $B$ is an analytic space, the analogous statement holds for holomorphic principal bundles and the sheaf $\underline{A}_B$ of holomorphic sections.
}

Fix a connected component $\fM^0_{\Lambda,G}$ of $\fM_{\Lambda,G}$.
Let $\underline{\Lambda}$ be the trivial local system over $\fM^0_{\Lambda,G}$ with fiber $\Lambda$
and define $\underline{G}$ similarly.
Consider the stack $\GG$ over $\fM^0_{\Lambda,G}$, which associates to 
each subset $U$ of $\fM^0_{\Lambda,G}$, open in the classical topology, the following category $\GG(U)$.
Objects of $\GG(U)$ are triples $(\pi,\eta,\psi)$, consisting of a family
$\pi:\X\rightarrow U$ of irreducible holomorphic symplectic manifolds, an isometric isomorphism of local systems
$\eta:R^2\pi_*\Integers \rightarrow \underline{\Lambda}_{U}$, and an isomorphism of local systems
$\psi:\SheafAut_0(\pi)\rightarrow \underline{G}_{U}$, such that 
the triple $(X_t,\eta_t,\psi_t)$ over a point $t$ of $U$ represents the isomorphism class parametrized by $t$ as a point of the coarse moduli space $\fM^0_{\Lambda,G}$.
The morphisms of $\GG(U)$ are isomorphisms of families, compatible with the trivializations of the two local systems. 

We observe next that $\GG$ is a gerbe over $\fM^0_{\Lambda,G}$ with band the trivial sheaf of groups $\underline{Z}$ with fiber the center $Z$ of $G$ (Lemma \ref{lemma-GG-is-a-gerbe}).
See section \ref{sec-gerbes} for the definitions. 
Equivalence classes of gerbes with band $\underline{Z}$, satisfying a technical property shared by $\GG$, are parametrized by $\check{H}^2(\fM^0_{\Lambda,G},Z)$
\cite{giraud}, \cite[Theorem 5.2.8]{brylinski}, \cite[Theorem 3.1]{moerdijk}, 
and Lemma \ref{lemma-GG-is-represented-by-a-cocycle} below. 
The existence of a universal family over $\fM^0_{\Lambda,G}$
is equivalent to the triviality of the gerbe, by \cite[III.2.1.1.2]{giraud}, and is thus also equivalent to the vanishing of the class 
$[\GG]\in \check{H}^2(\fM^0_{\Lambda,G},Z)$. 
Denote by $\phi_*[\GG]$ the image of $[\GG]$ in $\check{H}^2(\fM^0_\Lambda,Z)$ via the isomorphism $\phi:\fM_{\Lambda,G}^0\rightarrow \fM^0_\Lambda$.

Associated to a K\"{a}hler class $\omega$ on $X$ and a marking $\eta$ 
is a {\em twistor line} $Tw_{\omega,\eta}$ in $\fM^0_\Lambda$ through the marked pair $(X,\eta)$
\cite[1.13]{huybrechts-basic-results}. 
$Tw_{\omega,\eta}$ is isomorphic to a projective line $\PP^1$.
Finally, we show that the restriction homomorphism $\check{H}^2(\fM^0_\Lambda,Z)\rightarrow \check{H}^2(Tw_{\omega,\eta},Z)$
is an isomorphism (Lemma \ref{lemma-pullback-by-kappa-is-an-isomorphism}).
The image of $\phi_*[\GG]$ in 
$\check{H}^2(Tw_{\omega,\eta},Z)$ vanishes, by the existence of the twistor family over $Tw_{\omega,\eta}$. 
Hence, the class $[\GG]$ is trivial and 
Theorem \ref{thm-main} follows.

\smallskip

\hide{
\begin{rem}
\label{rem-non-fine-moduli-space}
Let $\pi:\X\rightarrow B$ be a family of irreducible holomorphic symplectic manifolds admitting trivializations $\eta:R^2\pi_*\Integers\rightarrow \underline{\Lambda}_B$ and $\psi:\SheafAut_0(\pi)\rightarrow \underline{G}_B$. 
Denote the associated classifying morphism by $\kappa_{(\pi,\eta,\psi)}:B\rightarrow \fM_{\Lambda,G}$. Then the set of isomorphism classes 
of triples $(\tilde{\pi},\tilde{\eta},\tilde{\psi})$, consisting of a family $\tilde{\pi}:\tilde{\X}\rightarrow B$ of irreducible holomorphic symplectic manifolds admitting trivializations $\tilde{\eta}:R^2\tilde{\pi}_*\Integers\rightarrow \underline{\Lambda}_B$ and 
$\tilde{\psi}:\SheafAut_0(\tilde{\pi})\rightarrow \underline{G}_B$ such that
$\kappa_{(\pi,\eta,\psi)}=\kappa_{(\tilde{\pi},\tilde{\eta},\tilde{\psi})}$, is an $\check{H}^1(B,Z)$-torsor, by 
Proposition \ref{lemma-GG-is-a-gerbe} below and 
\cite[Prop. 5.2.5]{brylinski}. 
In particular, if the center $Z$ of $G$ is not trivial, then the moduli space 
$\fM_{\Lambda,G}$ does {\em not} represent the functor $\F_{\Lambda,G}$ in Remark \ref{rem-fine-moduli-space}. 
Indeed, let $B$ be an analytic space with non-trivial $\check{H}^1(B,Z)$ and let $\kappa:B\rightarrow \fM_{\Lambda,G}$
be a constant morphism with image the isomorphism class of $(X_0,\eta_0,\psi_0)$. 
We get an $\check{H}^1(B,Z)$-torsor of isomorphism classes of triples $(\tilde{\pi},\tilde{\eta},\tilde{\psi})$, with 
$\tilde{\pi}:\tilde{\X}\rightarrow B$ isotrivial with fiber $X_0$, 
one of which is the trivial family with trivializations extending $\eta_0$ and $\psi_0$.
\end{rem}
}

Irreducible holomorphic symplectic manifolds of {\em $K3^{[n]}$-type} are those, which are deformation equivalent to the Hilbert scheme of length $n$ subschemes of a $K3$ surface. If $X$ is of $K3^{[n]}$-type, then any automorphism of $X$, which acts trivially on $H^2(X,\Integers)$, is the identity.
This follows for Hilbert schemes by a result of Beauville \cite{beauville},
and consequently also for their deformations, see
\cite[Sec. 2]{hassett-tschinkel-lagrangian-planes}. The automorphism group of every marked pair $(X,\eta)$, with $X$ of $K3^{[n]}$-type, is thus trivial, and Theorem \ref{thm-main} is known in this case. 

\begin{example}
Fix an integer $n\geq 2$.
Let $T$ be a two-dimensional compact complex torus, $T^{[n+1]}$ its Douady space of length $n+1$ subschemes,  
$T^{(n+1)}$ its $(n+1)$ symmetric product, 
and 
consider the fiber $K^{[n]}(T)$ over $0\in T$ of the composition
$T^{[n+1]} \rightarrow T^{(n+1)}\rightarrow T,$
where the left arrow is the Hilbert Chow morphism, and the right is summation. The fiber 
$K^{[n]}(T)$ is a $2n$-dimensional irreducible holomorphic symplectic manifold known as the
{\em generalized Kummer manifold} associated to $T$ \cite{beauville-varieties-with-zero-c-1}. 
Translation by points of $T$ of order $n+1$ induce automorphisms 
of $T^{[n+1]}$ leaving $K^{[n]}(T)$ invariant, as does multiplication by $-1$. These automorphisms generate
$\Aut_0(K^{[n]}(T))$, by \cite[Cor. 5]{BSN}. 
When $X$ is deformation equivalent to $K^{[n]}(T)$,
$\Aut_0(X)$ 
is thus isomorphic to the semidirect product $[\Integers/(n+1)\Integers]^{4}\rtimes \Integers/2\Integers$, where 
the non-trivial element of $\Integers/2\Integers$ acts on $[\Integers/(n+1)\Integers]^{4}$ via multiplication by $-1$
(see \cite[Theorem 3 and Corollary 5]{BSN}). 
The center of $\Aut_0(X)$ is trivial, if $n$ is even, and it is isomorphic to $[\Integers/2\Integers]^4$, if $n$ is odd.
\end{example}

\hide{
\begin{example}
Let $\Sigma$ be the subgroup of $T$ of points of order $n+1$. $\Sigma$ acts on 
$T$ by translations as well as on $K^{[n]}(T)$, as observed above. The quotient of $T\times K^{[n]}(T)$ by the diagonal action of $\Sigma$ is isomorphic to  $T^{[n+1]}$. Let $\pi:T^{[n+1]}\rightarrow T$ be the isotrivial family displayed in (\ref{eq-isotrivial-family}). 
$\Sigma$ is a subgroup of $G:=\Aut_0(K^{[n]}(T))$ and so the action of $\Sigma$ on $T$ lifts to an action on 
the trivial local system $\underline{G}_T$ over $T$, where the action on the fibers is via conjugation.
The local system $\SheafAut_0(\pi)$ over $T$ is the quotient of $\underline{G}_T$ by this action and is thus non-trivial. 
The local system $R^2\pi_*\Integers$ is trivial, and so a choice of an isometry $\eta_0:H^2(K^{[n]}(T),\Integers)\rightarrow \Lambda$ with
a fixed lattice $\Lambda$ yields a trivialization $\eta$ of $R^2\pi_*\Integers$ and the classifying morphism
$\kappa_{\pi,\eta}:T\rightarrow \fM_\Lambda$ is a constant morphism. 
Hence, the pullback via $\kappa_{\pi,\eta}$ of the universal family over $\fM_\Lambda$
is a trivial family.
While the moduli space $\fM_{\Lambda,G}$ represents the functor $\F_{\Lambda,G}$ for even $n$, 
by Remark \ref{rem-fine-moduli-space}, the moduli space $\fM_\Lambda$ only coarsely represents the functor $\F_\Lambda$. 
For odd $n\geq 3$ the moduli space $\fM_{\Lambda,G}$ only coarsely represents $\F_{\Lambda,G}$, 
by Remark \ref{rem-non-fine-moduli-space}. 
\end{example}
}

The group $\Aut_0(X)$ acts faithfully on the total cohomology ring 
$H^*(X,\Integers)$, when $X$ is deformation equivalent to a generalized Kummer manifold, by a result of Oguiso \cite[Theorem 1.3]{oguiso}. 
An alternative proof of Theorem \ref{thm-main} for this deformation type follows from Oguiso's result 
via the argument used in the proof of Lemma \ref{lemma-M-Lambda-G} below. 

The classification of irreducible holomorphic symplectic manifolds is an open problem. Two additional deformation types are known at present, one of  six dimensional manifolds \cite{ogrady-6} and one of ten dimensional manifolds
\cite{ogrady-10}. 

This work was motivated by the talk of Zhiyuan Li at the workshop
``Hyper-K\"{a}hler Manifolds, Hodge Theory, and Chow Groups''
at the  Tsinghua Sanya International Mathematics Forum in December 2016. In his talk Li surveyed 
consequences of the existence of universal families over Teichm\"{u}ller spaces to the study of 
cycles on moduli spaces of polarized irreducible holomorphic symplectic manifolds, generalizing previous work in the $K3$ surface case
\cite{BZMM}.

%
\section{Moduli spaces of marked pairs and triples}
\label{sec-moduli-spaces}

Fix a lattice $\Lambda$ isometric to the Beauville-Bogomolov-Fujiki lattice of some irreducible holomorphic symplectic manifold.
Set $\Lambda_\ComplexNumbers:=\Lambda\otimes_\Integers\ComplexNumbers$.
Let $\Omega_\Lambda$ be the period domain of irreducible holomorphic symplectic manifolds with 
Beauville-Bogomolov-Fujiki lattice $\Lambda$
\begin{equation}
\label{eq-period-domain}
\Omega_\Lambda:=\{\ell\in \PP(\Lambda_\ComplexNumbers) \ : \ (\ell,\ell)=0, \ \mbox{and} \ (\ell,\bar{\ell})>0\}.
\end{equation}
Fix a connected component $\fM_\Lambda^0$ of the moduli space of marked irreducible holomorphic symplectic manifolds.
The {\em period map} $P:\fM^0_\Lambda\rightarrow \Omega_\Lambda$ sends the isomorphism class of a marked pair $(X,\eta)$ to
$\eta(H^{2,0}(X))$.  Given a point 
$\ell\in\Omega_\Lambda$, denote by $\Lambda^{1,1}(\ell)$ the sublattice of $\Lambda$ orthogonal to $\ell$. 
The period map is a surjective local homeomorphism \cite{huybrechts-basic-results}. $P$ is generically injective and 
any two points in the same fiber of $P$ are inseparable, by Verbitsky's Global Torelli Theorem 
\cite{verbitsky,huybrechts-torelli}. 
A point $(X,\eta)$ in $\fM_\Lambda^0$ is a separated point, if and only if the K\"{a}hler cone of $X$ is equal to its {\em positive cone}, where the latter is the connected component of $\{\alpha\in H^{1,1}(X,\RealNumbers) \ : \ (\alpha,\alpha)>0\}$ containing the K\"{a}hler cone \cite{verbitsky,huybrechts-torelli} (see also \cite[Theorem 2.2 (4)]{markman-survey}). Consequently, 
if $\Lambda^{1,1}(\ell)$ is trivial, or cyclic spanned by a class $\lambda$ with $(\lambda,\lambda)\geq 0$,  
then the fiber $P^{-1}(\ell)$ consists of  a single separated point of  $\fM_\Lambda^0$
\cite[Corollaries 5.7 and 7.2]{huybrechts-basic-results}
(see also \cite[Theorem 2.2 (5)]{markman-survey}). 

\begin{lem}
\label{lemma-on-local-systems}
Every local system over $\fM_\Lambda^0$ is trivial.
\end{lem}

\begin{proof}
The period domain $\Omega_\Lambda$ is simply-connected, by 
\cite[Cor. 3]{huybrechts-period-domains}. Hence, it suffices to prove that every local system $\LB$  over $\fM_\Lambda^0$ 
is the pullback of a local system over $\Omega_\Lambda$ via the period map $P$. 
Choose a covering $\U:=\{U_i \ : \ i\in I\}$ of $\fM_\Lambda^0$ by simply connected open subsets $U_i$, such that 
$P$ restricts to each $U_i$ as a homeomorphism. Set $V_i:=P(U_i)$. 
We get the open covering $\V:=\{V_i \ : \ i\in I\}$ of $\Omega_\Lambda$, by the surjectivity of the period map.
Given an $n$-tuple $\vec{i}:=(i_1, i_2, \dots, i_n)\in I^n$, denote by $J_{\vec{i}}$ the set of connected components of 
$U_{\vec{i}}:=\cap_{k=1}^n U_{i_k}.$ Set $V_{\vec{i}}:=\cap_{k=1}^n V_{i_k}.$ 
The period map restricts to an open embedding of $U_{\vec{i}}$ into $V_{\vec{i}}$.
The complement of $P(U_{\vec{i}})$ in $V_{\vec{i}}$ is contained in the intersection of $V_{\vec{i}}$ with the countable union
of closed complex analytic hyperplanes $\Omega_\Lambda\cap \lambda^\perp$, as $\lambda$ varies in the set of 
primitive classes in $\Lambda$ with $(\lambda,\lambda)<0$, by the Global Torelli Theorem. Consequently, 
each connected component of $V_{\vec{i}}$ contains the image of a unique connected component of 
$U_{\vec{i}}$, by \cite[Lemma 4.10]{verbitsky}. We get a one-to-one correspondence between the set of connected components of $U_{\vec{i}}$ and $V_{\vec{i}}$. 

The restriction $\restricted{\LB}{U_i}$ of $\LB$ to $U_i$ is a trivial local system, for all $i\in I$, as $U_i$ is simply connected. 
Set $\Gamma(U_i,\LB):=H^0(U_i,\restricted{\LB}{U_i})$.
The evaluation homomorphism
$ev_t^i:\Gamma(U_i,\LB)\rightarrow \LB_t$ is an isomorphism, for every fiber $\LB_t$ of $\LB$ over a point $t$ of $U_i$. 
Given $\vec{i}\in I^2$ and $c\in J_{\vec{i}}$, denote by $U_c$ the corresponding connected component of $U_{\vec{i}}$. 
Given  points $t_i$ and $t'_i$ in $U_i$, parallel transport along any path from $t_i$ to $t'_i$ in $U_i$ is given by 
$ev^i_{t'_i}(ev^i_{t_i})^{-1}$.
The restriction homomorphism
$
\rho^{i_k}_c:\Gamma(U_{i_k},\LB)\rightarrow \Gamma(U_c,\LB)
$
is an isomorphism, for $k=1,2,$ and for every $c\in J_{\vec{i}}$.
Hence, the gluing of $\restricted{\LB}{U_{i_1}}$
and $\restricted{\LB}{U_{i_2}}$ along $U_c$ is determined by the isomorphism 
\[
f_{c,i_1,i_2}:=(\rho^{i_1}_c)^{-1}\circ \rho^{i_2}_c : \Gamma(U_{i_2},\LB)\rightarrow \Gamma(U_{i_1},\LB)
\]
in the sense that the following diagram is commutative  for a point $t_c$ in $U_c$.
\[
\xymatrix{
\Gamma(U_{i_2},\LB) \ar[r]^{\rho^{i_2}_c} 
\ar[dr]_{ev^{i_2}_{t_c}} \ar@/^3pc/[rr]_{f_{c,i_1,i_2}}&
\Gamma(U_c,\LB) \ar[d]^{ev^c_{t_c}} &
\Gamma(U_{i_1},\LB) \ar[dl]^{ev^{i_1}_{t_c}}
\ar[l]_{\rho^{i_1}_c}
\\
&\LB_{t_c}&
}
\]
Given a point $t\in U_{i_1}\cap U_{i_2}$, set $g_{i_1,i_2}(t):=f_{c,i_1,i_2}:\Gamma(U_{i_2},\LB)\rightarrow \Gamma(U_{i_1},\LB)$,
where $U_c$ is the connected component containing $t$. 
The transformation $g_{i_1,i_2}$ glues  the trivializations 
$\restricted{\LB}{U_{i_1}}\nolinebreak \cong\nolinebreak \underline{\Gamma(U_{i_1},\LB)}_{U_{i_1}}$
and $\restricted{\LB}{U_{i_2}}\nolinebreak \cong \nolinebreak \underline{\Gamma(U_{i_2},\LB)}_{U_{i_2}}$ along $U_{i_1,i_2}.$

Let $P_i:U_i\rightarrow V_i$ be the restriction of $P$ to $U_i$.
Denote by $P_i^*:\Gamma(V_i,P_{i,*}\restricted{\LB}{U_i})\rightarrow \Gamma(U_i,\restricted{\LB}{U_i})$ the natural isomorphism and let $P_{i,*}$ be its inverse. We get the isomorphisms
\[
\bar{f}_{c,i_1,i_2}:=P_{i_1,*}\circ
f_{c,i_1,i_2}
\circ P_{i_2}^*:
\Gamma(V_{i_2},P_{i_2,*}\restricted{\LB}{U_{i_2}})\rightarrow \Gamma(V_{i_1},P_{i_1,*}\restricted{\LB}{U_{i_1}}).
\]
The latter determines a gluing 
of $P_{i_1,*}\restricted{\LB}{U_{i_1}}$ and $P_{i_2,*}\restricted{\LB}{U_{i_2}}$
along the connected component $V_c$ of $V_{i_1,i_2}$, for every $c\in J_{\vec{i}}$, hence along all the connected components of $V_{i_1,i_2}$. Denote by $\bar{g}_{i_1,i_2}$
the gluing of the trivializations 
$P_{i_k,*}\restricted{\LB}{U_{i_k}}\cong \underline{\Gamma(V_{i_k},P_{i_k,*}\restricted{\LB}{U_{i_k}})}_{V_{i_k}}$, $k=1,2$, 
along $V_{i_1,i_2}$. These gluing transformations satisfy the co-cycle condition
$\bar{g}_{i,k}=\bar{g}_{i,j}\bar{g}_{j,k}$, since each $\bar{g}_{i,j}$ pulls back to $g_{i,j}$ and the $g_{i,j}$'s satisfy the co-cycle condition. 
Let $\overline{\LB}$ be the local system over $\Omega_\Lambda$ determined by the covering $\V$ and the gluing transformations $\bar{g}_{i_1,i_2}$. 
Then the restriction of $P^*\overline{\LB}$ to $U_i$ is naturally identified with that of $\LB$ and each gluing transformation 
$P^*\bar{g}_{i_1,i_2}$ restricts to $g_{i_1,i_2}$. Hence, $\LB$ is isomorphic to $P^*\overline{\LB}$.
\end{proof}

We keep the notation of the introduction. 
Fix a group $G$ isomorphic to the group $\Aut_0(X)$ of some irreducible holomorphic symplectic manifold $X$. 
The moduli space $\fM_{\Lambda,G}$ is constructed by gluing all Kuranishi families
$\pi:\X\rightarrow U$, each endowed with a choice of an isometric trivialization $\eta:R^2\pi_*\Integers\rightarrow \underline{\Lambda}_U$
and a trivialization $\psi:\Aut_0(\pi)\rightarrow \underline{G}_U$. The construction is completely analogous to that of 
$\fM_\Lambda$ in \cite[Prop. 4.3]{huybrechts-torelli}. The following is an immediate corollary of Lemma \ref{lemma-on-local-systems}.


\begin{lem}
\label{lemma-M-Lambda-G}
The forgetful morphism $\phi:\fM_{\Lambda,G}\rightarrow \fM_\Lambda$ restricts to each connected component 
$\fM^0_{\Lambda,G}$ of $\fM_{\Lambda,G}$ as an isomorphism onto the corresponding connected component 
$\fM^0_\Lambda$ of $\fM_\Lambda$. The set of connected components of $\fM_{\Lambda,G}$ over 
$\fM^0_\Lambda$ forms a torsor under the group of outer automorphisms of $G$.
\end{lem}

\begin{proof}
%
Let $Ad_g\in \Aut(G)$ be conjugation by $g\in G$.
$G$ acts on the set $\Isom(\Aut_0(X),G)$, of group isomorphisms from $\Aut_0(X)$ onto $G$, by 
the action $\psi \mapsto Ad_g\circ \psi$.
Similarly, $f\in\Aut_0(X)$ acts on 
$\Isom(\Aut_0(X),G)$ via $\psi\mapsto \psi\circ Ad_{f^{-1}}$. The set of orbits for the two actions coincide and we get a natural identification 
$\Isom(\Aut_0(X),G)/G=\Isom(\Aut_0(X),G)/\Aut_0(X),$ so we denote both orbit spaces by $\Out(\Aut_0(X),G)$.
Over $\fM_\Lambda$ we have the local system $\SheafOut_G$, whose fiber over the point corresponding to the isomorphism class
$(X,\eta)$ is $\Out(\Aut_0(X),G)$.
If $f$ is an automorphism of a marked pair $(X,\eta)$,
then $f$ belongs to $\Aut_0(X)$. Hence, the local system  $\SheafOut_G$ is well defined.

The moduli space $\fM_{\Lambda,G}$ is simply the total space of the local system $\SheafOut_G$.
The local system $\SheafOut_G$ restricts to a trivial local system over each connected component $\fM_\Lambda^0$, by Lemma
\ref{lemma-on-local-systems}. 
Hence, connected components of $\fM_{\Lambda,G}$ are simply global sections of $\SheafOut_G$. 
\end{proof}

Given an abelian group $A$ and a topological space $S$, denote by 
$\check{H}^i(S,A)$ the $i$-th \v{C}ech cohomology of $S$ with coefficients in $A$. 

\begin{lem}
\label{lemma-pullback-by-the-period-map-is-an-isomorphism}
The pull back homomorphism 
$P^*:\check{H}^i(\Omega_\Lambda,A)\rightarrow \check{H}^i(\fM^0_\Lambda,A)$ is an isomorphism, for every abelian group $A$.
\end{lem}

\begin{proof}
We keep the notation of the proof of Lemma \ref{lemma-on-local-systems}. 
The simply connected open sets $U$ of $\fM^0_\Lambda$, such that $P$ restricts to $U$ as a homeomorphism, form a basis for the topology of $\fM^0_\Lambda$. 
Given an open covering $\U=\{U_i\}_{i\in I}$, consisting of such open sets, we get an open covering 
$\V:=  \{V_i:=P(U_i)\}_{i\in I}$ of $\Omega_\Lambda$. 
The covering $\U$ is a refinement of the covering $\{P^{-1}(V_i)\}_{i\in I}$, so
we get a pullback homomorphism 
\[
P^*:\check{H}^i(\V,A)\rightarrow \check{H}^i(\U,A).
\]
Every open covering of $\Omega_\Lambda$ admits a refinement by an open covering $\V$ as above. 
Hence, it suffices to prove that the pullback homomorphism displayed above 
is an isomorphism, for such a covering $\U$ of $\fM^0_\Lambda$ and the induced covering $\V$ of $\Omega_\Lambda$. 
The degree $n$ group $\C^n(\U,A)$ in the \v{C}ech complex is 
$\oplus_{\vec{i}\in I^{n+1}}\oplus_{c\in J_{\vec{i}}}\Gamma(U_c,A)$
and $\Gamma(U_c,A)=A$. The natural bijection between the connected components of $U_{\vec{i}}$ and $V_{\vec{i}}$,
observed in the proof of Lemma \ref{lemma-on-local-systems}, implies that 
 the pullback homomorphism $P^*$ induces an isomorphism of \v{C}ech complexes. 
\end{proof}


%
\section{Gerbes}
\label{sec-gerbes}

\begin{defi}
A {\em fibered category} (or a {\em presheaf of categories}) $\FF$ over a topological space $S$ consists of the following.
\begin{enumerate}
\item
A category $\FF(U)$ for each open subset $U$ of $S$.
\item
A functor $i^*:\FF(U)\rightarrow \FF(V)$, for each inclusion $i:V\rightarrow U$ of open sets.
\item
A natural isomorphism $\tau_{ij}:(ij)^*\rightarrow j^*i^*$ for each composition  $W\RightArrowOf{j} V\RightArrowOf{i}U$
of inclusions.
\end{enumerate}
The natural isomorphisms are assumed to satisfy a natural associativity property \cite{giraud},\cite[Def. \nolinebreak 5.2.1]{brylinski}. See also
\cite[Def. 2.1]{moerdijk}.
\end{defi}

\begin{defi}
A  fibered category $\FF$ over  a topological space $S$ is called a {\em prestack} if, for any pair of objects $a, b$ of $\FF(U)$, the presheaf $\SheafHom(a,b)$  over $U$, associating to an inclusion $i:V\rightarrow U$ of an open subset $V$ the set $\Hom_{\FF(V)}(i^*a,i^*b)$, is a sheaf. A prestack $\FF$ is called a {\em stack}, if it satisfies an additional descent condition for objects
\cite{giraud,brylinski} (see also \cite[Def. 2.6]{moerdijk}).
\end{defi}

\begin{defi}
\label{def-gerbe}
(\cite{giraud}, \cite[Def. 5.2.4]{brylinski}, or \cite[Def. 3.1]{moerdijk}).
A {\em gerbe}  over a topological space $S$ is a stack $\GG$ satisfying the following properties.
\begin{enumerate}
\item
\label{def-gerbe-groupoid}
Each category $\GG(U)$ is a groupoid.
\item
\label{non-emptiness-axiom}
$S$ admits a covering by open sets $U$, such that $\GG(U)$ has an object. 
\item
\label{transitivity-axiom}
Given objects $a, b$ of $\GG(U)$, any point $x\in U$ has an open neighborhood $V \subset U$, 
such that $\Hom_{\GG(V)}(i^*a,i^*b)$ is non-empty, where $i:V\rightarrow U$ is the inclusion.
\end{enumerate}
\end{defi}

Let $\GG$ be a gerbe over a topological space $S$, such that for every open set $U$
and every object $a$ of $\GG(U)$, the group $\Aut_{\GG(U)}(a)$ is abelian. 
We define next a sheaf $\Z$ of abelian groups called the {\em band} of $\GG$. 
For a more general definition, for arbitrary gerbes, see \cite{giraud} or \cite[Def. 3.2]{moerdijk}.
We can choose an open covering $\U:=\{U_\alpha\}$ and an object $a_\alpha$ in $\GG(U_\alpha)$, yielding a
sheaf $\SheafAut(a_\alpha)$ over each $U_\alpha$, by Axiom (\ref{non-emptiness-axiom})
of Definition \ref{def-gerbe}. These sheaves admit a canonical gluing to a sheaf $\Z$ of abelian groups as follows.
Choose a covering $\{U_{\alpha\beta}^\xi\}$ of each $U_{\alpha\beta}$, such that there exists an isomorphism 
$f_{\alpha\beta}^\xi$ from the restriction $\restricted{(a_\beta)}{U_{\alpha\beta}^\xi}$
of $a_\beta$ to $U_{\alpha\beta}^\xi$ to the restriction $\restricted{(a_\alpha)}{U_{\alpha\beta}^\xi}$ of $a_\alpha$.
Such a covering exists, by Axioms (\ref{def-gerbe-groupoid}) and  (\ref{transitivity-axiom}) of Definition \ref{def-gerbe}. 
Conjugation by $f_{\alpha\beta}^\xi$ induces a sheaf isomorphism 
\begin{equation}
\label{eq-gluing-transformations-of-the-band}
\lambda_{\alpha\beta}^\xi:\SheafAut\left(\restricted{(a_\beta)}{U_{\alpha\beta}^\xi}\right)
\rightarrow 
\SheafAut\left(\restricted{(a_\alpha)}{U_{\alpha\beta}^\xi}\right).
\end{equation}
Now $\lambda_{\alpha\beta}^\xi$ is independent of the choice of $f_{\alpha\beta}^\xi$, since any other choice 
is of the form $h f_{\alpha\beta}^\xi$, for an element $h$
of the abelian group $\Aut_{\GG(U_{\alpha\beta}^\xi)}\left(\restricted{(a_\alpha)}{U_{\alpha\beta}^\xi}\right)$. 
In particular, $\lambda_{\alpha\beta}^\xi$ and $\lambda_{\alpha\beta}^{\tilde{\xi}}$ agree over overlaps 
$U^\xi_{\alpha\beta}\cap U^{\tilde{\xi}}_{\alpha\beta}$ and define an isomorphism
\[
\lambda_{\alpha\beta}:\SheafAut\left(\restricted{(a_\beta)}{U_{\alpha\beta}}\right)
\rightarrow 
\SheafAut\left(\restricted{(a_\alpha)}{U_{\alpha\beta}}\right).
\]
The isomorphisms $\lambda_{\alpha\beta}$ satisfy the co-cycle condition 
$\lambda_{\alpha\beta}\lambda_{\beta\gamma}\lambda_{\gamma\alpha}=1$
in $\Aut\left(\SheafAut\left(\restricted{(a_\alpha)}{U_{\alpha\beta\gamma}}\right)\right)$, as the left hand side is an inner automorphism of an abelian group.
Hence, the sheaves $\SheafAut(a_\alpha)$ glue to a sheaf $\Band(\GG)$ of abelian groups, called the {\em band} of the gerbe $\GG$. 
The isomorphism class of the sheaf $\Band(\GG)$ is independent of the choice of the covering $\{U_\alpha\}$ and of the objects $\{a_\alpha\}$.

\begin{defi}
Let $\Z$ be a sheaf of abelian groups. A {\em gerbe with band} $\Z$ is a gerbe $\GG$, all of whose objects over all open subsets have abelian automorphism groups, together with an isomorphism $\theta:\Band(\GG)\rightarrow\Z$ of sheaves of abelian groups.
\end{defi}

There is a notion of {\em equivalence} $\phi:\FF\rightarrow \GG$ between stacks $\FF$ and $\GG$ over a topological space $X$, which consists of equivalences of categories $\phi(U):\FF(U)\rightarrow \GG(U)$, for every open subset $U$ of $X$,
satisfying a natural descent condition \cite[Page 199]{brylinski} and \cite[Def. 2.3, and Rem. (i) page 12]{moerdijk}. A {\em morphism} $\phi:(\GG,\theta)\rightarrow(\GG',\theta')$ between
two gerbes $(\GG,\theta)$, $(\GG',\theta')$  with band $\Z$ is an equivalence $\phi:\GG\rightarrow\GG'$
of fibered categories satisfying 
$\theta'\phi=\theta$. Two gerbes $(\GG,\theta)$, $(\GG',\theta')$ with band $\Z$ are said to be {\em equivalent} if there exists a third gerbe $(\FF,\rho)$ with band $\Z$
and morphisms from $(\FF,\rho)$ to each of $(\GG,\theta)$ and $(\GG',\theta')$ \cite[Paragraph after Def. 3.2]{moerdijk}.
The set of equivalence classes of gerbes with abelian band $\Z$ is denoted by $\Gerbes(X,\Z)$. If the pair $(X,\Z)$ has the property that every open covering of $X$ admits a refinement $\U:=\{U_\alpha\}_{\alpha\in I}$, such that 
$H^1(U_{\alpha\beta},\Z)=0$, for all $\alpha,\beta\in I$, then there is a bijective correspondence 
\begin{equation}
\label{eq-Giraud-isomorphism}
\Gerbes(X,\Z)\cong \Check{H}^2(X,\Z),
\end{equation}
with the second \v{C}ech cohomology group with coefficients in $\Z$,
by \cite{giraud}, \cite[Theorem 5.2.8]{brylinski} and  by\footnote{Note that the additional vanishing of the  cohomologies
$H^2(U_{\alpha\beta},Z(K_\alpha)|U_{\alpha\beta})$ 
and $H^1(U_{\alpha\beta\gamma},Z(K_\alpha)|U_{\alpha\beta\gamma})$
with coefficients in the center of a non-abelian band $K$, assumed in \cite[Theorem 3.1]{moerdijk}, is not needed as it is only used in the Remark preceding \cite[Def. 3.2]{moerdijk} to lift a cocycle of outer automorphisms to a cocycle of automorphisms. In our case of an abelian band $\Z$ such a lift is not needed.}  \cite[Theorem 3.1]{moerdijk}. 
The vanishing $H^1(U_{\alpha\beta},\Z)=0$ is satisfied for {\em good open coverings} (for which all non-empty finite intersections are contractible), and abelian local systems $\Z$, and so the above property holds for $(X,\Z)$, whenever $X$ is a Hausdorff manifolds and $\Z$ is an abelian local system. 

\begin{defi}
\label{def-gerbes-with-a-cocycle}
Denote by
$\Gerbes_0(X,\Z)$ the set of equivalence classes of gerbes $\GG$ with band $\Z$ 
such that $X$ admits an open covering
$\U:=\{U_{\alpha}\}_{\alpha\in I}$ with objects $a_\alpha\in \GG(U_\alpha)$, such that 
$\SheafHom(a_\alpha|U_{\alpha\beta},a_\beta|U_{\alpha\beta})$ is the trivial 
$\SheafAut(a_\alpha|U_{\alpha\beta})$-torsor, for all $\alpha,\beta\in I$ with non-empty $U_{\alpha\beta}$.
\end{defi}

If $X$ is a Hausdorff manifold and $\Z$ is an abelian local system, then $\Gerbes_0(X,\Z)=\Gerbes(X,\Z)$,
by the vanishing $H^1(U_{\alpha\beta},\Z)=0$ for a good covering.
Note that once the property in the above Definition holds for an open covering, it holds also for every refinement of this covering.
The above bijection (\ref{eq-Giraud-isomorphism}) associates to a gerbe $(\GG,\theta)$ a \v{C}ech cohomology class represented by a $2$-cocycle 
$g_{\alpha\beta\gamma}:=\theta_\alpha(f_{\alpha\beta}f_{\beta\gamma}f_{\alpha\gamma}^{-1})\in\Z(U_{\alpha\beta\gamma})$
associated to choices of isomorphisms $f_{\alpha\beta}:a_\beta|U_{\alpha\beta}\rightarrow a_\alpha|U_{\alpha\beta}$
\cite[Prop. 3.1]{moerdijk}. The proof of \cite[Theorem 3.1]{moerdijk} constructs, more generally, a bijection $\Gerbes_0(X,\Z)\cong \Check{H}^2(X,\Z)$, for any topological space $X$ and a sheaf of abelian groups $\Z$.

Fix a lattice $\Lambda$ and a finite group $G$.
Let $B$ be an analytic space and $\kappa:B\rightarrow \fM_{\Lambda,G}$ a morphism. 
Consider the fibered category $\kappa^{-1}\GG$ over $B$, which associates to an open set $U\subset B$ the following category
$\kappa^{-1}\GG(U)$. Objects of $\kappa^{-1}\GG(U)$ are triples $(\pi,\eta,\psi)$ consisting of a family $\pi:\X\rightarrow U$ of irreducible
holomorphic symplectic manifolds, an isometric trivialization $\eta:R^2\pi_*\Integers \rightarrow \underline{\Lambda}_U$, and a
trivialization $\psi:\SheafAut_0(\pi)\rightarrow \underline{G}_U$, such that for each point  $u\in U$ the triple
$(X_u,\eta_u,\psi_u)$ represents the isomorphism class corresponding to the point $\kappa(u)$
of $\fM_{\Lambda,G}$. A morphism in
$\Hom_{\kappa^{-1}\GG(U)}((\pi,\eta,\psi),(\tilde{\pi},\tilde{\eta},\tilde{\psi}))$ is an isomorphism 
$f:\X\rightarrow \tilde{\X}$, satisfying $\pi=\tilde{\pi}f$, such that the induced isomorphism of local systems 
$f^*:R^2\tilde{\pi}_*\Integers\rightarrow R^2\pi_*\Integers$ satisfies $\tilde{\eta}=\eta f^*$ and such that 
$\tilde{\psi}=\psi Ad_f$, where $Ad_f:\SheafAut_0(\tilde{\pi})\rightarrow \SheafAut_0(\pi)$ is the isomorphism of sheaves of groups induced by conjugation by $f$.

\begin{lem}
\label{lemma-GG-is-a-gerbe}
The fibered category  $\kappa^{-1}\GG$ is a gerbe over $B$ with band $\underline{Z}_B$, where $Z$ is the center of $G$.
\end{lem}

\begin{proof}
It suffices to prove that the fibered category $\GG$ over $\fM_{\Lambda,G}$, associated to the identity morphism, is a gerbe with band $\underline{Z}$, 
as the more general $\kappa^{-1}\GG$ described above is simply the inverse image of $\GG$ via $\kappa$ and is thus a gerbe with band $\underline{Z}_B$, 
by \cite[Prop. 5.2.6]{brylinski}.
Property (\ref{def-gerbe-groupoid}) of Definition \ref{def-gerbe} holds, by definition of morphisms in $\GG$.
The vector space $H^0(X,TX)$ vanishes for every irreducible holomorphic symplectic manifold $X$, and
the dimension of $H^1(X,TX)$ is the Hodge number $h^{1,1}(X)$. Hence, the versal Kuranishi family of $X$ is universal 
\cite[Ch. I, Theorems (10.3) and (10.5)]{BHPV}. Every point $t:=(X_0,\eta_0)$ of $\fM_\Lambda$ admits a simply connected open neighborhood $U$, over which we have a universal family of deformations of $X_0$, as  $\fM_\Lambda$ is constructed by gluing such families \cite[Prop. 4.3]{huybrechts-torelli}. The same holds for $\fM_{\Lambda,G}$, 
by Lemma \ref{lemma-M-Lambda-G}.
Let $b$ be a point of $\fM_{\Lambda,G}$ and let $(X_b,\eta_b,\psi_b)$ represent the isomorphism class $b$. 
Choose a simply connected open neighborhood $U\subset \fM_{\Lambda,G}$ of $b$, over which
we have a universal family $\pi:\X\rightarrow U$ of deformations of $X_b$. The local systems 
$R^2\pi_*\Integers$ and $\SheafAut_0(\pi)$ are trivial, since $U$ is simply connected, and so $\eta_b$ and $\psi_b$
extend to an isometric trivialization $\eta:R^2\pi_*\Integers \rightarrow \underline{\Lambda}_U$ and 
a trivialization $\psi:\SheafAut_0(\pi)\rightarrow \underline{G}_U$. 
We get the  object $(\pi,\eta,\psi)$ in $\GG(U)$.
Property
(\ref{non-emptiness-axiom})  of Definition \ref{def-gerbe} follows. 
Property (\ref{transitivity-axiom})   of Definition \ref{def-gerbe} holds, since the universal family is universal for each of its fibers
(see \cite[Ch. I, Theorems (10.3) and (10.6) and the Remark following (10.6)]{BHPV}).

Given an open subset $U$ of $\fM_{\Lambda,G}$ and an 
object $a:=(\pi,\eta,\psi)$ in $\GG(U)$, the sheaf $\SheafAut_{\GG(U)}(a)$ is isomorphic to 
the center of $\SheafAut_0(\pi)$ and 
$\psi$ restricts to an isomorphism from the center of $\SheafAut_0(\pi)$ onto $\underline{Z}_U$. 
The isomorphisms induced by the $\psi$'s are compatible with the gluing transformations
$\lambda_{\alpha\beta}^\xi$ in (\ref{eq-gluing-transformations-of-the-band}), by definition of morphisms in $\GG$. 
Hence, the band of $\GG$ is isomorphic to $\underline{Z}$.
\end{proof}

\begin{lem}
\label{lemma-GG-is-represented-by-a-cocycle}
The gerbe $\GG$ over $\fM_{\Lambda,G}$ satisfies the property in Definition \ref{def-gerbes-with-a-cocycle} and is hence represented by a class $[\GG]$ in $\Check{H}^2(\fM_{\Lambda,G},\underline{Z})$.
\end{lem}

\begin{proof}
It suffices to consider the restriction of $\GG$ to a connected component $\fM^0_{\Lambda,G}$ of $\fM_{\Lambda,G}$.
Set $\tilde{P}:=P\circ\phi:\fM^0_{\Lambda,G}\rightarrow \Omega_\Lambda$. Let $\U:=\{U_i\}_{i\in I}$ be an open covering of $\fM^0_{\Lambda,G}$, such that 
there exists over each $U_i$ a universal family $\pi_i:\X_i\rightarrow U_i$ with trivializations $\eta_i$ of $R^2\pi_{i,*}\Integers$ and $\psi_i$ of $\Aut_0(\pi_i)$ and 
$\tilde{P}$ restricts to each $U_i$ as a homeomorphism $\tilde{P}:U_i\rightarrow V_{j(i)}:=\tilde{P}(U_i)$,
where $j:I\rightarrow J$ is a function and 
$\{V_j\}_{j\in J}$ is a good covering of $\Omega_\Lambda$. 
Such a covering exists, since $\Omega_\Lambda$ is a Hausdorff manifold, and so every open covering of $\Omega_\Lambda$ admits a refinement by a good covering.
Then  $V_{j(i_1)}\cap V_{j(i_2)}$ is simply connected, for all $i_1,i_2\in I$, but  the intersection $U_{i_1}\cap U_{i_2}$ need not be simply connected. 

Denote by $\X_{i_1}|V_{j(i_1)}\cap V_{j(i_2)}$ the restriction to $V_{j(i_1)}\cap V_{j(i_2)}$ of the pullback via $\tilde{P}^{-1}:V_{j(i_1)}\rightarrow U_{i_1}$ 
of  $\X_{i_1}$. 
Define $\X_{i_2}|V_{j(i_1)}\cap V_{j(i_2)}$
similarly. 
The morphisms $\pi_{i_1}$ and $\pi_{i_2}$ are weakly K\"{a}hler, by \cite[Lemma 4.14]{bakker-lehn}, and hence so is the fiber product 
\[
\pi_{i_1i_2}:\X_{i_1i_2}\rightarrow V_{j(i_1)}\cap V_{j(i_2)}
\]
of $\X_{i_1}|V_{j(i_1)}\cap V_{j(i_2)}$ and $\X_{i_2}|V_{j(i_1)}\cap V_{j(i_2)}$
over $V_{j(i_1)}\cap V_{j(i_2)}$, by \cite[(2.2)]{fujiki-1982}. 
Let $\D\rightarrow V_{j(i_1)}\cap V_{j(i_2)}$ be the relative Douady space of $\pi_{i_1i_2}$.
We refer to \cite[Def. 2.3]{bakker-lehn} for the definition of {\em weakly K\"{a}hler}. All we need is that it is the property required to conclude that every irreducible component  of $\D$ is proper over  $V_{j(i_1)}\cap V_{j(i_2)}$, by the main theorem of \cite{fujiki-1984}.
Let $\D_0$ be the (finite) union of irreducible components which contain points parametrizing the graph of an isomorphism 
$f_t:\X_{i_1,t}\rightarrow \X_{i_2,t}$, $t\in U_{i_1}\cap U_{i_2}$, compatible with the markings $\eta_{i_k,t}$ and $\psi_{i_k,t}$, $k=1,2$.
Then $\D_0$ intersects the inverse image of $\tilde{P}(U_{i_1}\cap U_{i_2})$ in a
$\underline{Z}$-torsor over $\tilde{P}(U_{i_1}\cap U_{i_2})$ (where we identify the restriction of $\underline{Z}$ to $U_{i_1}$ with the center of $\SheafAut_0(\pi_{i_1})$ via $\psi_{i_1}$). 
Equivalently, $\D_0$ is a principal $Z$-bundle over $\tilde{P}(U_{i_1}\cap U_{i_2})$.
A priori this $Z$-bundle extends to a $Z$-bundle of graphs of isomorphisms
away from a closed analytic subset of $V_{j(i_1)}\cap V_{j(i_2)}$, by the properness\footnote{More precisely, we use here also the properness of the universal subscheme of $\D_0\times \X_{i_1i_2}$ over $\D_0$ and the fact that the locus of points of the universal subscheme, where it is not an isomorphism, is closed.} of $\D_0$ over $V_{j(i_1)}\cap V_{j(i_2)}$. 
This closed analytic subset is $V_{j(i_1)}\cap V_{j(i_2)}\setminus \tilde{P}(U_{i_1}\cap U_{i_2})$, by definition of $\fM_{\Lambda,G}$.

Let $C$ be a smooth and connected Riemann surface in $V_{j(i_1)}\cap V_{j(i_2)}$, which intersects 
$\tilde{P}(U_{i_1}\cap \nolinebreak U_{i_2})$. The inverse image $\tilde{C_0}$ of $C_0:=C\cap \tilde{P}(U_{i_1}\cap U_{i_2})$ in $\D_0$ extends, abstractly, to a branched cover $\tilde{C}$ of $C$. 
Note that $C$ intersects $V_{j(i_1)}\cap V_{j(i_2)}\setminus \tilde{P}(U_{i_1}\cap U_{i_2})$ in a union of isolated points, as the intersection is a closed proper analytic subset, by the  discussion in the previous paragraph.
The triviality of the 
$Z$-bundle $\D_0$ is equivalent to that of the associated representation of the fundamental group of $U_{i_1}\cap U_{i_2}$ in the abelian group $Z$. 
The latter would follow once we show that for every such $C$ the morphism $\tilde{C}\rightarrow C$
is in fact unramified.
Associated to every point $\tilde{t}$ of $\tilde{C}\setminus \tilde{C}_0$, over $t\in [C\setminus C_0]$, is a limiting cycle
$\Gamma_{\tilde{t}}+\sum_k Y_{\tilde{t},k}$
in $\X_{i_1,t}\times \X_{i_2,t}$,  such that $\Gamma_{\tilde{t}}$ is the graph of a bimeromorphic map as well as  the unique summand dominating the fibers $\X_{i_1,t}$ and
$\X_{i_2,t}$, by the proof of  \cite[Theorem 4.3]{huybrechts-basic-results}. 

The group $Z$
acts on $\tilde{C}_0$ and this action is free and transitive on all fibers of $\tilde{C}_0$ over $C_0$.
Hence, $Z$
acts transitively on
the set of limiting cycles associated to points in the same fiber of $\tilde{C}\rightarrow C$. We claim that the latter action is free.
Indeed, if $g$ is a bimeromorphic map from $\X_{i_1,t}$ to $\X_{i_2,t}$ and $f_1$, $f_2$ are elements of
$\Aut_0(\X_{i_1,t})$, such that $gf_1=gf_2$, then $f_1=f_2$.
Hence, $\tilde{C}$ is unramified over $C$.
\end{proof}

\begin{rem}
\label{rigidification} Jenia Tevelev pointed out to me that 
gerbes arise naturally from the construction of the 
{\em rigidification of a stack} (see \cite[Sec. 5.1]{ACV}). In the above lemma the coarse moduli space $\fM_{\Lambda,G}$ is the rigidification of the stack $\GG$. The construction of rigidification of a stack was used by Gorchinskiy and Viviani to reprove the classical result that  a universal family exists over the open subset of the coarse moduli space of hyperelliptic curves of genus $g$ without extra automorphisms apart from the hyperelliptic involution, if and only if $g$ is odd \cite[Prop. 4.7]{GV}. 
\end{rem}

%
\section{Vanishing of the cohomology class of a gerbe}
\label{sec-triviality-of-a-gerbe}

Fix a lattice $\Lambda$ isometric to the Beauville-Bogomolov-Fujiki lattice of some irreducible holomorphic symplectic manifold.
Let $W$ be a three dimensional positive definite subspace of $\Lambda_\RealNumbers$. 
Let $Q_W\subset \PP(W_\ComplexNumbers)$ be the conic of isotropic lines in $W_\ComplexNumbers:=W\otimes_\RealNumbers\ComplexNumbers$.
Denote by $\iota:Q_W\rightarrow \Omega_\Lambda$ the inclusion into the period domain (\ref{eq-period-domain}).

\begin{lem}
\label{lemma-the-period-domain-is-homotopic-to-the-2-sphere}
The inclusion $\iota$ is a homotopy equivalence.
\end{lem}

\begin{proof}
Let $Gr_+(3,\Lambda_\RealNumbers)$ be the Grassmannian of positive definite three dimensional subspaces of 
$\Lambda_\RealNumbers$.  The identity component $SO_+(\Lambda_\RealNumbers)$ of the special orthogonal group
acts transitively on $Gr_+(3,\Lambda_\RealNumbers)$. The stabilizer of $W$ in $SO_+(\Lambda_\RealNumbers)$ 
is $SO(W)\times SO(W^\perp)$ realizing $Gr_+(3,\Lambda_\RealNumbers)$ as the quotient 
$SO_+(\Lambda_\RealNumbers)/[SO(W)\times SO(W^\perp)]$. Now $SO(W)\times SO(W^\perp)$ is a maximal compact subgroup of $SO_+(\Lambda_\RealNumbers)$ and so the latter is topologically the product of the former and a Euclidean space, by Cartan's Theorem
\cite[Theorem 2]{mostow}. Hence, $Gr_+(3,\Lambda_\RealNumbers)$ is contractible. 

A point of the period domain  $\Omega_\Lambda$ corresponds to an isotropic line $\ell\subset \Lambda_\ComplexNumbers$, such that
$[\ell+\bar{\ell}]\cap\Lambda_\RealNumbers$ is a positive definite two-dimensional subspace $V$ of $\Lambda_\RealNumbers$. 
The natural identification of the two dimensional real vector spaces $\ell$ and $V$ endows the latter with an orientation. This construction 
identifies the period domain $\Omega_\Lambda$ with the Grassmannian of oriented positive definite two-dimensional subspaces of $\Lambda_\RealNumbers$ \cite[Sec. 4.2]{huybrechts-period-domains}.
Let $I\subset \Omega_\Lambda\times Gr_+(3,\Lambda_\RealNumbers)$ be the incidence correspondence. 
The fiber of  the projection $q:I\rightarrow Gr_+(3,\Lambda_\RealNumbers)$ over $W$ is $Q_W$.
The inclusion $\tilde{\iota}:Q_W\rightarrow I$ of the fiber is a homotopy equivalence, since $q$
is a fibration over a contractible base. 
The fiber of the projection $p:I\rightarrow \Omega_\Lambda$ over an oriented two dimensional subspace 
$V\subset \Lambda_\RealNumbers$ is the projectivization of the positive cone in the subspace $V^\perp$
of signature $(1,b_2-3)$, i.e. a hyperbolic space. Hence, $p$ is a homotopy equivalence, being a fibration with contractible fibers. 
We conclude that $\iota=p\circ \tilde{\iota}$ is the composition of two homotopy equivalences and so is such as well.
\end{proof}

Let $X$ be an irreducible holomorphic symplectic manifold and $\omega$ a K\"{a}hler class on $X$. 
Set $V:=[H^{2,0}(X)\oplus H^{0,2}(X)]\cap H^2(X,\RealNumbers)$
and set $W:=V+\RealNumbers\omega$. Then $W$ is a positive definite three dimensional subspace of 
$H^2(X,\RealNumbers)$. Let $I$ be the complex structure of $X$.
There exists a unique Ricci flat hermetian metric $g$ on $X$, whose imaginary part is a K\"{a}hler form representing the class $\omega$, by Yau's proof of the Calabi conjecture \cite{beauville-varieties-with-zero-c-1}. Furthermore, there exists two additional complex structures $J$ and $K$, covariantly constant  with respect to the Riemannian metric which is the real part of the hermetian metric $g$, such that $IJ=K$. The identity, $I$, $J$, and $K$ span a subalgebra of endomorphisms of the real tangent bundle, which is isomorphic to the algebra $\HH$ of quaternions. The two-sphere 
\[
Tw_\omega := \{aI+bJ+cK \ : \ a^2+b^2+c^2=1\},
\]
of purely imaginary unit quaternions, consists of integrable complex structures. 
The Riemannian metric and each of these complex structures $I_t\in Tw_\omega$ determine a K\"{a}hler form $\omega_t$ on
the manifold $\overline{X}$ underlying  $X$, hence a Hodge structure. Denote by $X_t$ the complex manifold  
$(\overline{X},I_t)$. 
The map 
$\kappa:Tw_\omega\rightarrow \PP(H^2(\overline{X},\ComplexNumbers))$, sending $I_t$ to $H^{2,0}(X_t)$, is a diffeomorphic embedding
whose image is $Q_W$ \cite[1.17]{huybrechts-basic-results}. Endow $Tw_\omega$ with the complex structure of $Q_W$, which is isomorphic to the complex projective line. We get
a complex structure on the smooth manifold $\X:=\overline{X}\times Tw_\omega$, such that the projection 
\begin{equation}
\label{eq-twistor-family}
\pi:\X\rightarrow Tw_\omega 
\end{equation}
is holomorphic, and the fiber of $\pi$ over $I_t\in Tw_\omega$ is isomorphic to 
$(\overline{X},I_t)$ \cite[Sec. 3(F)]{hklr}. The above family is known as the {\em twistor family} associated to the K\"{a}hler form $\omega$. 

Choose a marking $\eta_0:H^2(X,\Integers)\rightarrow \Lambda$. It extends uniquely to an isometric trivialization
$\eta:R^2\pi_*\Integers\rightarrow \underline{\Lambda}_{Tw_\omega}$, since $Tw_\omega$
is simply connected. We get an embedding 
\[
\kappa_{\pi,\eta}:Tw_\omega\rightarrow \fM^0_\Lambda
\]
into the connected component of $\fM_\Lambda$ containing $(X,\eta_0)$.

\begin{lem}
\label{lemma-pullback-by-kappa-is-an-isomorphism}
The pullback homomorphism $\kappa_{\pi,\eta}^*: \check{H}^*(\fM^0_\Lambda,A)\rightarrow \check{H}^*(Tw_\omega,A)$
is an isomorphism, for every abelian group $A$.
\end{lem}

\begin{proof}
The composition $P\circ\kappa_{\pi,\eta}:Tw_\omega\rightarrow \Omega_\Lambda$ with the period map $P$ 
is the embedding of the base $Tw_\omega$ of the twistor family as the conic 
$Q_{\eta(W)}$ of isotropic lines in the complexification of 
the positive three dimensional subspace $\eta(W)$ of $\Lambda_\RealNumbers$.
Set $\iota:=P\circ\kappa_{\pi,\eta}$. The pullback 
$\iota^*\nolinebreak :\nolinebreak 
\check{H}^*(\Omega_\Lambda,A)\rightarrow \check{H}^*(Tw_\omega,A)$ is an isomorphism, 
by Lemma \ref{lemma-the-period-domain-is-homotopic-to-the-2-sphere}, and the pullback
$P^*\nolinebreak : \nolinebreak \check{H}^*(\fM^0_\Lambda,A)\rightarrow \check{H}^*(\Omega_\Lambda,A)$ 
is an isomorphism, by Lemma \ref{lemma-pullback-by-the-period-map-is-an-isomorphism}. Hence, 
$\kappa_{\pi,\eta}^*=\iota^*\circ(P^*)^{-1}$ is an isomorphism as well.
\end{proof}

\begin{proof}[{\bf Proof of Theorem \ref{thm-main}}]
Choose some twistor family (\ref{eq-twistor-family}). 
The local system $\SheafAut_0(\pi)$ over $Tw_\omega$  is trivial, as the latter is simply connected. 
Choose a trivialization $\psi:\SheafAut_0(\pi)\rightarrow \underline{G}_{Tw_\omega}$, let
$
\kappa_{\pi,\eta,\psi}:Tw_\omega\rightarrow \fM_{\Lambda,G}
$
be the classifying morphism, and let $\fM_{\Lambda,G}^0$ be the connected component 
containing its image. Let $Z$ be the center of $G$.  Then 
$\kappa_{\pi,\eta,\psi}^*:\check{H}^*(\fM^0_{\Lambda,G},Z)\rightarrow \check{H}^*(Tw_\omega,Z)$
is an isomorphism, by Lemmas \ref{lemma-pullback-by-kappa-is-an-isomorphism} and \ref{lemma-M-Lambda-G}.

Let $\GG$ be the gerbe over $\fM^0_{\Lambda,G}$ associated to the identity morphism from 
$\fM^0_{\Lambda,G}$ to itself as in Lemma \ref{lemma-GG-is-a-gerbe}. 
Let $\FF:=\kappa_{\pi,\eta,\psi}^{-1}\GG$ be the gerbe over $Tw_\omega$ associated to 
the morphism $\kappa_{\pi,\eta,\psi}$ as in Lemma \ref{lemma-GG-is-a-gerbe}. 
Equivalence classes of gerbes over $Tw_\omega$ with band a sheaf $\Z$ of abelian groups are in bijection with
cohomology classes in $\check{H}^2(Tw_\omega,\Z)$ \cite{giraud}, \cite[Theorem 5.2.8]{brylinski}. 
The equivalence class $[\FF]$ of  $\FF$ in $\check{H}^2(Tw_\omega,Z)$ vanishes, since 
we have the object $(\pi,\eta,\psi)$ in $\FF(Tw_\omega)$. 
Similarly, the gerbe $\GG$ is represented by a class $[\GG]$ in $\check{H}^2(\fM^0_{\Lambda,G},Z)$,
by Lemma \ref{lemma-GG-is-represented-by-a-cocycle}.
The morphism
$\kappa_{\pi,\eta,\psi}$ pulls back the class $[\GG]$ to $[\FF]$.
Hence, $\kappa_{\pi,\eta,\psi}^*([\GG])$ vanishes. We conclude that the class $[\GG]$ vanishes, 
since $\kappa_{\pi,\eta,\psi}^*$ is an isomorphism. Consequently, $\GG(\fM^0_{\Lambda,G})$ has an object, by
\cite[III.2.1.1.2]{giraud}.

The set of isomorphism classes of objects of $\GG(\fM^0_{\Lambda,G})$ is a torsor for the group 
$\check{H}^1(\fM^0_{\Lambda,G},Z)$, by \cite[Prop. 5.2.5]{brylinski}.
The latter is the trivial group, by Lemmas \ref{lemma-pullback-by-kappa-is-an-isomorphism} and \ref{lemma-M-Lambda-G}.
Hence, $\GG(\fM^0_{\Lambda,G})$ has a unique object, up to isomorphism. The forgetful morphism 
$\phi$ restricts to an  isomorphism $\phi:\fM^0_{\Lambda,G}\rightarrow \fM^0_\Lambda$ 
(Lemma \ref{lemma-M-Lambda-G}), transferring the object of
$\GG(\fM^0_{\Lambda,G})$ to a $\Lambda$-marked family $(\pi:\X\rightarrow \fM^0_\Lambda,\eta)$ over $\fM^0_\Lambda$. 

It remains to prove the universal property of $(\pi,\eta)$. 
Let $(\tilde{\pi}:\tilde{\X}\rightarrow B,\tilde{\eta})$ be a $\Lambda$-marked family over a connected analytic space $B$, let $\kappa:B\rightarrow \fM_\Lambda$ be the classifying morphism, and let $\fM_\Lambda^0$ be the connected component containing its image. 
The family $\pi:\X\rightarrow \fM^0_\Lambda$ is locally universal,
as it restricts to a universal Kuranishi family over an open neighborhood of each point of $\fM^0_\Lambda$. 
Hence, the $\Lambda$-marked families $(\tilde{\pi},\tilde{\eta})$ and $(\kappa^*(\pi),\kappa^*(\eta))$ are locally isomorphic. 
Let $\P:=\SheafIsom((\tilde{\pi},\tilde{\eta}),(\kappa^*(\pi),\kappa^*(\eta)))$ be the local system of isomorphisms of the
two families, which are compatible with the markings. Then $\P$ is an $\SheafAut_0(\tilde{\pi})$ torsor and 
$(\kappa^*(\pi),\kappa^*(\eta))$ is isomorphic to $(\tilde{\pi},\tilde{\eta})\times\P$.
\end{proof}

\begin{rem}
\label{rem-rigidified}
Let $X$ be an irreducible holomorphic symplectic manifold, 
$\Diff(X)$ its diffeomorphism group, and $\Diff_0(X)$ the subgroup of elements isotopic to the identity. 
An easy sufficient criterion for the existence of a universal family over the connected component of Teichm\"{u}ller space 
is the triviality of the intersection $\Aut(X)\cap \Diff_0(X)$ for every $X$ in this connected component \cite[Sec. 1.4]{catanese}. Such $X$ is called {\em rigidified} \cite[Def. 12]{catanese}.
\end{rem}

\begin{rem}
\label{rem-principal-G-bunle-invariant}
Let $(\tilde{\pi}:\tilde{\X}\rightarrow B,\tilde{\eta})$ be a $\Lambda$-marked family  over a connected analytic space $B$, let $\kappa:B\rightarrow \fM_\Lambda^0$ be the classifying morphism, and let 
$(\kappa^*(\pi):\kappa^*\X\rightarrow B,\kappa^*(\eta))$ be the pullback of the universal family. 
We can choose a trivialization 
$\psi:\SheafAut_0(\kappa^*(\pi))\rightarrow \underline{G}_B$, by Lemma \ref{lemma-M-Lambda-G}.
Thus, the orbit of the isomorphism class of the principal $G$-bundle 
$\P_{(\tilde{\pi},\tilde{\eta})}:=\SheafIsom((\kappa^*(\pi),\kappa^*(\eta)),(\tilde{\pi},\tilde{\eta}))$, under the group $\Out(G)$ of outer automorphisms, is an invariant of the 
$\Lambda$-marked family $(\tilde{\pi},\tilde{\eta})$. 
Let us reconstruct this invariant more directly under the assumption that the fibers of $\tilde{\pi}$ are rigidified (Remark \ref{rem-rigidified}). 
Fix a fiber $X$ of $\tilde{\pi}$. Let $\Diff'(X)$ be the subgroup of $\Diff(X)$ fixing the connected component of Teichm\"{u}ller space containing $X$ and acting trivially on $H^2(X,\Integers)$. 
$\Gamma:=\Diff'(X)/\Diff_0(X)$ is a subgroup of the mapping class group $\Diff(X)/\Diff_0(X)$.
The natural homomorphism
$h:\Aut_0(X)\rightarrow \Gamma$ is an isomorphism. Indeed, $h$ is injective, as $X$ is assumed to be rigidified, and surjective by
\cite[Theorem 4.26(iii) and Cor. 4.31]{verbitsky} 
(the groups $\Gamma$ and $\Aut_0(X)$ are denoted by  $G_I$ and $K_I$ in \cite{verbitsky}). 
The $\Lambda$-marked family 
$(\tilde{\pi},\tilde{\eta})$ is differentiably locally trivial. Hence, there exists an open covering $\{U_i\}$ of $B$ and trivializations 
$\restricted{\tilde{X}}{U_i}\cong X\times U_i$, such that the gluing transformations are given by continuous maps from $U_i\cap U_j$ to $\Diff'(X)$. These gluing transformations yield a principal  $\Diff'(X)$-bundle, hence a principal $\Gamma$-bundle, hence a principal $\Aut_0(X)$-bundle, which coincides with the one described above.
\end{rem}

%

{\bf Acknowledgements:}
This work was partially supported by a grant  from the Simons Foundation (\#427110). 
I thank Zhiyuan Li for his interesting talk, and for 
asking whether universal families exist over Teichm\"{u}ller spaces of holomorphic symplectic  manifolds, at the workshop
``Hyper-K\"{a}hler Manifolds, Hodge Theory, and Chow Groups''
at the  Tsinghua Sanya International Mathematics Forum in December 2016.
I thank the organizers, Radu Laza, Kieran O'Grady, and Claire Voisin for the invitation to this interesting and instructive workshop. I thank Daniel Huybrechts and Sukhendu Mehrotra for helpful communications and for their comments on an earlier draft of this note. I thank
Jenia Tevelev for Remark \ref{rigidification}. I thank the referee for his numerous insightful comments and suggestions.

\end{document}